\documentclass[11pt, a4 paper]{amsart} 

\usepackage[psamsfonts]{amssymb} 
\usepackage{amsfonts,amsmath,mathabx}
\usepackage{color}

\usepackage{amsthm, amssymb}
\usepackage{amsfonts}

\newtheorem{thmA}{Theorem}

\newtheorem{corA}[thmA]{Corollary}

\numberwithin{equation}{section} 

\newtheorem{theorem}{Theorem}[section]
\newtheorem{prop}[theorem]{Proposition}
\newtheorem{proposition}[theorem]{Proposition}
\newtheorem{lemma}[theorem]{Lemma}
\newtheorem{corollary}[theorem] {Corollary}

\newtheorem*{theorem*}{1-2-3 Theorem} 

\theoremstyle{remark}
\newtheorem{remark}[theorem]{Remark}

\theoremstyle{definition}

\def\Z{\mathbb Z}
\def\N{\mathbb N}

\def\Q{\mathbb Q}

\def\G{\Gamma}

\def\e{\varepsilon}

\def\-{\overline}
\def\wh{\widehat}

\def\G{\Gamma} 

\def\G{\Gamma}

\def\<{\langle}
\def\>{\rangle}


\begin{document}

\title[Homology of groups and echoes of Baumslag]{The homology of groups, 
profinite completions, and echoes of Gilbert Baumslag}

% author  information
\author[Martin R. Bridson]{Martin R.~Bridson}
\address{Martin R.~Bridson\\
Mathematical Institute\\
Andrew Wiles Building\\
ROQ, Woodstock Road\\
Oxford\\
European Union}
\email{bridson@maths.ox.ac.uk}

%\date{25 July 2019, 16 Aug 2019, typos Dec 2019}


\begin{abstract}  We present novel constructions concerning the homology of finitely generated groups. 
Each construction draws on ideas of Gilbert Baumslag.
 There is a finitely presented acyclic group $U$ such that $U$ has no proper subgroups of finite index and
every finitely presented group can be embedded in $U$.
There is no algorithm that can determine whether or not a finitely
presentable subgroup of a residually finite, biautomatic group is perfect. 
For every recursively presented abelian group $A$ there 
exists a pair of groups $i:P_A\hookrightarrow G_A$ such that $i$ induces an isomorphism of profinite completions, where $G_A$  
is a torsion-free biautomatic group that is residually finite and  superperfect,
while $P_A$ is a finitely generated group with $H_2(P_A,\Z)\cong A$. 
\end{abstract}
\maketitle

\centerline{\em For Gilbert Baumslag, in memoriam}

\section{Introduction}                 

Gilbert Baumslag took a great interest in the homology of groups. Famously,  with 
Eldon Dyer and Chuck Miller \cite{BDM} he proved  
that an arbitrary sequence of countable
abelian groups $(A_n)$, with $A_1$ and $A_2$ finitely generated,
will arise as the homology sequence $H_n(G,\Z)$ of some
finitely presented group $G$, provided that the $A_n$ can be 
described in an untangled recursive manner. This striking result built on Gilbert's earlier work with
Dyer and Alex Heller \cite{BDH}. 
A variation on 
arguments from  \cite{BDM} and \cite{BDH} yields the following result, which will be useful
in our study of profinite completions of discrete groups.
Recall that a group $G$ is termed {\em{acyclic}} if $H_n(G,\Z)=0$ for
all $n\ge 1$. 

\begin{thmA}\label{t:thm1}
 There is a finitely presented acyclic group $U$ such that
\begin{enumerate}
\item $U$ has no proper subgroups of finite index;
\item every finitely presented group can be embedded in $U$.
\end{enumerate}
\end{thmA}

A recursive presentation $(X\,|\, R)_{\rm{Ab}}$ of an abelian group is said to be
{\em untangled} if the set $R$ is a basis for the subgroup $\<R\>$
of the free abelian group generated by $X$.  
The following corollary can be deduced  from
Theorem \ref{t:thm1} using the Baumslag-Dyer-Miller construction; see Section \ref{s:homology}.

\begin{corA}\label{c:cor} Let $\mathcal{A}=(A_n)_n$ be a sequence of abelian
groups, the first of which is finitely generated. If the $A_n$ are 
given by a recursive sequence of recursive presentations, each of which
is untangled, then there is a finitely presented group $Q_{\mathcal{A}}$ with no proper subgroups of finite index and $H_n(Q_{\mathcal{A}}, \Z)\cong A_{n-1}$ for all $n\ge 2$.
\end{corA}

In \cite{BR}, Gilbert and
Jim Roseblade used homological arguments to prove that every finitely presented 
subdirect product of two finitely generated free groups is either 
free or of finite index. 
This insight was the germ for a large and immensely rich
body of work concerning residually free groups and subdirect products of hyperbolic
groups, with the homology of groups playing a central role. The
pursuit of these ideas
has occupied a substantial part of my professional life \cite{BM, BHMS1, BHMS2, BHMS3}
and also commanded much of Gilbert's attention in the latter part of his 
career \cite{bbms1, bbms2, AlgGeom, bbhm}.  A cornerstone of this programme
is the 1-2-3 Theorem, which Gilbert and I proved in our second paper with 
Chuck Miller and Hamish Short \cite{bbms2}. 

The proof of the following theorem provides a typical example of the utility
of the 1-2-3 Theorem. It extends the theme of \cite{bbms1, bbms2}, which demonstrated
the wildness that is to be found among the finitely presented subgroups of automatic
groups. It also reinforces the point made in \cite{BW} about the necessity of including
the full input data in the effective version of the 1-2-3 Theorem \cite{BHMS3}.
The proof that the ambient biautomatic group is residually finite relies on deep work
of  Wise \cite{wiseSC, wiseBIG} and Agol \cite{agol} as well as Serre's insights into the connection between  residual finiteness
and cohomology with finite coefficient modules \cite[Section I.2.6]{serre}.

\begin{thmA}\label{t:perfect} There is no algorithm that can determine whether or not a finitely presentable subgroup of a residually finite, biautomatic group is perfect.
\end{thmA}

To prove this theorem, we construct a recursive sequence $(G_n,H_n)$
where $G_n$ is biautomatic
group given by a finite presentation $\< X\mid R_n\>$ and
$H_n<G_n$ is the subgroup generated by a finite set $S_n$ of words in the generators $X$;
the cardinality of $S_n$ and $R_n$ does not vary with $n$. The construction ensures that $H_n$ is finitely presentable, but a consequence of the theorem is that there is no algorithm that can use this knowledge to construct an 
explicit presentation of $H_n$. An artefact of the construction is that 
each $G_n$ has a finite classifying space $K(G_n,1)$.

Besides picking up on the themes of Gilbert mentioned above, Theorem \ref{t:perfect}
also resonates with a longstanding theme in his work, often pursued in partnership with
Chuck Miller, whereby one transmits undecidability
phenomena from one context to another in group theory by building groups
that encode the appropriate phenomenon by means of graphs of groups, wreath
products, directly constructed presentations, or whatever else one can dream up. This 
is already evident in his early papers, particularly \cite{BN}.

I have discussed three themes from Gilbert Baumslag's oeuvre:
{\bf{(i)}} decision problems and their transmission through explicit constructions;
{\bf{(ii)}} homology of groups; and
{\bf{(iii)}} subdirect products of free and related groups. To these
I add two more (neglecting others):
{\bf{(iv)}} a skill for constructing explicit groups that illuminate important
phenomena, inspired in large part by his formative interactions with Graham Higman,
Bernhard Neumann and Wilhelm Magnus; and {\bf{(v)}}
 an enduring interest in residual finiteness and nilpotence, 
 with an associated interest in profinite and pronilpotent  completions of groups.

In the 1970s Gilbert and his students,
particularly Fred Pickel \cite{pick, GPS}, explored the extent to which
finitely generated, residually
finite groups are determined by their finite images (equivalently,
their profinite completions -- see Section \ref{s:pf}). He maintained
a particular focus on residually nilpotent groups,  
motivated in particular by a desire to find the right context
in which to understand parafree groups. In his survey \cite{gilb-nilp}
he writes:``More than 35 years ago, Hanna Neumann asked whether free groups
can be characterised in terms of their lower central series. Parafree
groups grew out of an attempt to answer her question." 
It is a theme
that he returned to often; see \cite{gilb-nilp}. 
I was drawn to the study of profinite completions later, by 
Fritz Grunewald \cite{BG}. As I have become increasingly
absorbed by it, Gilbert's illuminating 
examples and provocative questions have been invigorating. 

I shall present one result concerning profinite completions here
and further results in the sequel to this paper \cite{mrb:next}.

The proof of the following
result combines a refinement of Corollary \ref{c:cor} and parts of the proof of Theorem \ref{t:perfect} 
with a somewhat involved spectral sequence argument.
It extends arguments from  Section 6 of
\cite{BReid} that were developed to answer questions posed by Gilbert in \cite{gilb-nilp}.

Recall that the profinite completion $\widehat G$ of a group $G$ is the inverse
limit of the directed system of finite quotients of $G$. 
A {\em{Grothendieck pair}} \cite{groth}
is a pair of residually finite
groups $\iota:A\hookrightarrow B$ such that the induced map of
profinite completions $\hat{\iota}:\wh{A}\to\wh{B}$ is an isomorphism. Recall also that a group
$G$ is termed {\em{superperfect}} if $H_1(G,\Z)=H_2(G,\Z)=0$.

\begin{thmA}\label{t:thmC} For every recursively presented abelian group $A$ there exists a Grothendieck pair
$P_A\hookrightarrow G_A$ where $G_A$  is a torsion-free biautomatic group that is residually-finite, superperfect and has a finite classifying space, while
$P_A$ is finitely generated with $H_2(P_A,\Z)\cong A$.
\end{thmA}

Note that $A$ need not be finitely generated here; for example, $A$ might be the group of
additive rationals $\Q$, or the direct sum of the cyclic groups $\Z/p\Z$ of all
prime orders.

The diverse background material that we require for the main results is gathered in Section \ref{s2}. 

\section{Preliminaries}\label{s2}

I shall assume that the reader is familiar with the basic theory 
of homology of groups  (\cite{bieri}
and \cite{brown} are excellent references) and the definitions of small cancellation
theory \cite{LS} and hyperbolic groups \cite{gromov}.  Recall that a classifying space $K(G,1)$ for a discrete group $G$
is a CW-complex with fundamental group $G$ and contractible universal cover.
$H_n(G,\Z) = H_n(K(G,1),\Z)$. One says that $G$ is of {\em type $F_n$} if there
is a classifying space $K(G,1)$ with finite $n$-skeleton. Finite generation 
is equivalent to type $F_1$ and finite presentability is equivalent to type $F_2$.

\subsection{Fibre Products}
Associated to a short exact sequence of groups $1\to N\to G\overset{\eta}\to Q\to 1$
one has the {\em fibre product} 
$$
P = \{(g,h) \mid \eta(g)=\eta(h)\} < G\times G.
$$
The restriction to $P$ of the projection $G\times G\to 1\times G$ 
has kernel $N\times 1$ and can be split
by sending $(1,g)$ to $(g,g)$. Thus $P \cong N\rtimes G$ where the
action is by conjugation in $G$.

\begin{theorem*}[\cite{bbms2}]\label{t:123}
Let $1\to N\to G\overset{\eta}\to Q\to 1$ be a short exact sequence of groups. If $N$ is finitely
generated, $G$ is finitely presented, and $Q$ is of type $F_3$, then the
associated fibre product $P<G\times G$ is finitely presented.
\end{theorem*}

The {\em effective 1-2-3 Theorem}, proved in \cite{BHMS2}, 
provides an algorithm that, given the following data, will construct
a finite presentation for $P$: a finite  presentation
$G=\<A\mid S\>$ is given, with a finite generating set for $N$ (as words in the
generators $A$), a finite presentation $\mathcal{P}$ for $Q$, a word defining  
 $\eta(a)$ for each  $ a\in A$, and a set of generators for $\pi_2\mathcal{P}$ as
a $\Z Q$-module. 

The proof of Theorem \ref{t:perfect} shows that one cannot dispense with this
last piece of data, while Theorem \ref{t:thmC} shows that the 1-2-3 Theorem would
fail if one assumed only that $Q$ was finitely presented. 

By definition, a generating set $A$ for $G$ defines an epimorphism $\mu:F\to G$,
where $F$ is the free group on $A$. We can choose a different presentation
$Q=\<A\mid R\>$ such that the identity map on $A$ defines the
composition $\eta\circ \mu: F\to Q$. The following lemma is easily checked.

\begin{lemma}\label{l:genP}
With the above notation, the fibre product $P<G\times G$ is generated by the image of
$
\{ (a,a),\, (r,1) \mid a\in A,\, r\in R\} \subset F\times F.$
\end{lemma}

We shall also need an observation that is useful when computing with the LHS spectral sequences
associated to  $1\to N\to G\to Q\to 1$ and to  $1\to N\times 1 \to P\to 1\times G\to 1$. In the first case, 
the term $H_0(Q,H_1N)$ arises,
which by definition is the group of coinvariants for the
action of $G$ on $N$ by conjugation, i.e.~$N/[N,G]$. The second spectral sequence
contains the term $H_0(G,H_1N)$; here the action of $g\in G$ is induced
by conjugation of $(g,g)$ on $N\times 1 < G\times G$, so  
$H_0(G,H_1N)$ is again $N/[N,G]$. More generally, because the action of $G$ on  $N$ is the same
in both cases  we have:

\begin{lemma}\label{l:sameH_0} In the context described above,
$H_0(Q,H_kN)\cong H_0(G,H_kN)$ for all $k\ge 0$.
\end{lemma}  

\subsection{Universal central extensions}  

A {\em central extension} of a group $Q$ is
a group $\widetilde Q$ equipped with a homomorphism $\pi:\widetilde{Q}\to Q$ whose kernel is central
in $\widetilde Q$. Such an extension is {\em universal} if given any other central extension
$\pi': E \to Q$ of $Q$, there is a unique homomorphism $f : \widetilde Q \to  E$ such that
$\pi'\circ f = \pi$. 

The standard reference for universal central extensions is 
\cite{milnor} pp. 43--47. The
properties that we need here are the following,
which all follow easily from standard facts (see \cite{mb:karl}
for details and references).

\begin{proposition}\label{p:uce} $\ $
\begin{enumerate}
\item $Q$ has a universal central extension $\widetilde Q\to Q$
if and only if $Q$ is perfect. (If it exists, $\widetilde Q\to Q$ is unique up to
isomorphism over $Q$.) 
\item There is a short exact sequence
$$1 \to  H_2(Q,\Z) \to \widetilde{Q}\to Q\to 1.$$
\item $H_1(\widetilde Q,\Z) = H_2(\widetilde Q,\Z)=0$.
\item If $Q$ has no non-trivial finite quotients, then neither does $\widetilde Q$.
%(6) If G is finitely presented then so is ˜ G.
\item For $k\ge 2$, if $Q$ is of type $F_k$ then so is $\widetilde{Q}$.
\item If  $Q$ has a compact 2-dimensional classifying space $K(Q, 1)$ then 
$\widetilde{Q}$ is torsion-free and has a compact
classifying space.
\end{enumerate}
\end{proposition}

The following result is Corollary 3.6 of \cite{mb:karl}; the proof relies on an argument
due to Chuck Miller.

\begin{proposition}\label{p:tilde}
 There is an algorithm that, given a finite presentation $\<A\mid {R}\>$
of a perfect group $G$, will output a finite presentation $\< A\mid \overline{R}\>$ 
defining a group $\widetilde{G}$ such that  the identity map on the set $A$
induces the universal central extension $\widetilde G\to G$.
Furthermore, $|\overline{R}| = |A|(1 + |{R}|)$.
\end{proposition}

\subsection{Applications of the Lyndon-Hochshild-Serre Spectral Sequence}\label{s:LHS}

 Besides the Mayer-Vietoris sequence,
the main tool that we draw on in our calculations 
of homology groups is the Lyndon-Hochshild-Serre
spectral sequence associated to a short exact sequence
of groups  $1\to N \to G\to Q\to 1$.   The $E^2$ page of this
spectral sequence is $E^2_{pq}=H_p(Q,H_q(N,\Z))$, and 
the sequence converges to $H_n(G,\Z)$; see \cite{brown}, p.171. 
A particularly useful region of the spectral sequence is the corner of the first
quadrant, from which one can isolate the 5-term exact sequence
\begin{equation}\label{5term}
H_2(G,\Z)\to H_2(Q,\Z)\to H_0(Q,\, H_1(N,\Z)) \to H_1(G,\Z) \to H_1(Q,\Z)\to 0.
\end{equation}
From this we immediately have:
\begin{lemma}\label{l0}
Let $1\to N\to G\to Q\to 1$ be a short exact sequence of groups. 
If $H_1(G,\Z)=H_2(G,\Z)=0$, then $H_2(Q,\Z)\cong H_0(Q,H_1N)$. 
\end{lemma}

The following calculations with the LHS spectral sequence will be needed in the
proofs of our main results.  

\begin{lemma} \label{l:acylic-fibre} If $1\to N\to G\to Q\to 1$
is exact and $N$ is acyclic, then $G\to Q$
induces an isomorphism $H_n(G,\Z)\to H_n(Q,\Z)$
for every $n$.
\end{lemma}

\begin{proof} In the 
LHS spectral sequence, the only non-zero entries
on the second page are $E^2_{n0}=H_n(Q,\Z)$, so  $E^2=E^{\infty}$ 
and  $H_n(G,\Z)\to E^\infty_{n0}= H_n(Q,\Z)$ is an isomorphism.
\end{proof}

In the following lemmas, all homology groups have coefficients in the trivial module $\Z$
unless stated otherwise.

\begin{lemma}\label{l1}
Let $1\to N\to B\overset{\eta}\to C\to 1$ be a short exact sequence of groups. 
If $H_1N= H_2B=0$
and $\eta_*: H_3B\to H_3C$ is the zero map, then $H_0(C, H_2N)\cong H_3C$.
\end{lemma}

\begin{proof} The hypothesis $H_1N=0$ implies that on the $E^2$-page
of the LHS spectral sequence, the terms in the second row $E^2_{*1}$ are all zero.
Thus all of the differentials emanating from the bottom two rows of the $E^2$-page
are zero, so $E^3_{p0}=E^2_{p0}$ for all $p\in\N$ and $E^3_{0q}=E^2_{0q}$ for $q\le 2$.
Hence the only non-zero differential emanating from place $(3,0)$ 
is on the $E^3$-page, and this is $d_3: H_3C\to H_0(C,H_2N)$.
The kernel of $d_3$ is $E^\infty_{30}$, the image of $\eta_*: H_3B\to H_3C$, which we have assumed to be zero. And the cokernel of $d_3$ is $E^\infty_{02}$, which injects into
$H_2B$, which we have also assumed is zero. Thus $d_3: H_3C\to H_0(C,H_2A)$ is an
isomorphism.
\end{proof}

\begin{lemma}\label{l2}  Let $P= A\rtimes B$. If $H_1A=H_2B=0$, then
$H_2P\cong H_0(B,H_2A)$. 
\end{lemma}

\begin{proof} By hypothesis, on the $E^2$-page of 
the LHS spectral sequence for $1\to A\to P\to B\to 1$,
the only non-zero term $E^2_{pq}$ with $p+q=2$ is $E^2_{02}=H_0(B,H_2A)$.
It follows that $H_2P \cong E^\infty_{02}=E^4_{02}$.
And since $E^2_{21}=H_2(B,H_1A)=0$, we also have $E^3_{02}=E^2_{02}=H_0(B,H_2A)$.
As $B$ is a retract of $P$, for every $n$ the natural map $H_nP\to H_nB$ is
surjective, so all differentials emanating from the bottom row of the spectral
sequence are zero. In particular, $d_3: H_3B\to H_0(B,H_2A)$ is the zero map,
and hence $E^4_{02}=E^3_{02}=H_0(B,H_2A)$.
\end{proof}

\begin{lemma}\label{l3}
Let $1\to N\to B\overset{\eta}\to C\to 1$ be a short exact sequence of groups. 
Suppose that
$H_1N$ is finitely generated, $H_1B=H_2C=0$ and $C$ has no non-trivial
finite quotients. Then $H_1N=0$.
\end{lemma}

\begin{proof} As $H_1N$ is finitely generated, its automorphism group is residually finite. Thus, since $C$ has no finite quotients, the action of $C$ on $H_1N$
induced by conjugation in $B$ must be trivial and $H_0(C,H_1N)=H_1N$.
From the LHS spectral sequence we isolate the exact sequence
$H_2C \to H_0(C, H_1N) \to H_1B$. The first and last groups are zero by hypothesis,
so $ H_1N = H_0(C,H_1N) =0$.
\end{proof}

\subsection{An adapted version of the Rips construction}\label{s:rips}

Eliyahu Rips discovered a remarkably elementary
construction  \cite{rips} that has proved to be enormously
useful in the exploration of the subgroups of hyperbolic and related groups.
There are many refinements
of his construction in which extra properties are imposed on the group constructed.
The following version is well adapted to our needs.

\begin{prop}\label{rips1}
There exists an algorithm that, given an integer $m\ge 6$ and a 
finite presentation $\mathcal Q \equiv  \< X\mid R\>$ of a group $Q$,
will construct a finite presentation $\mathcal P \equiv \< X\cup \{a_1,a_2\}\mid \widetilde{R}\cup V\>$
for a group $\G$ so that
\begin{enumerate}
\item $N:=\<a_1,a_2\>$ is normal in $\G$, 
\item $\G/N$ is isomorphic to $Q$, %the group with presentation $\mathcal Q$,
\item $\mathcal P$ satisfies the small cancellation condition $C'(1/m)$, and 
\item $\G$ is perfect if $Q$ is perfect.
\end{enumerate}
\end{prop}

\begin{proof} The original argument of Rips \cite{rips} proves all but the last item. In his argument,
one chooses a set of reduced words $\{u_r \mid r\in R\}\cup\{ v_{x,i,\e}\mid x\in X,\, i=1,2,\, 
\e=\pm 1\}$   in the free group on $\{a_1,a_2\}$, all of length at least $m\, \max \{|r| : r\in R\}$,
 that satisfies $C'(1/m)$. Then $\widetilde{R}=\{ru_r\mid r\in R\}$   and $V$ consists of the relations
 $xa_i x^{-1}v_{x,i,\e}$ with $x\in X^\e,\, i=1,2,$ and $\e=\pm 1$.
 Such a choice can be made algorithmic (in many different ways). 

To ensure that (4) holds, one chooses
 the words $v_{x,i,\e}$ to
have exponent sum $0$ in $a_1$ and $a_2$. Such a choice ensures
that the image of $N$ in $H_1\G$ is trivial, so if $\G/N\cong Q$
is perfect then so is $\G$. One way to arrange that the exponent sums are
zero is by a simple substitution: choose $\tilde R\cup V$ as above and then
replace each occurrence of $a_1$  by $a_1a_2a_1^{-2}a_2^{-1}a_1$
and each occurrence of $a_2$ by $a_2a_1a_2^{-2}a_1^{-1}a_2$. If the original 
construction is made so that the presentation is $C'(1/{5m})$ then this modified
presentation will be $C'(1/m)$.
\end{proof}

\begin{remark}\label{r:special} Dani Wise \cite{wiseSC} proved that metric small cancellation
groups can be cubulated and, building on work of Wise
\cite{wiseBIG}, Agol \cite{agol} proved that cubulated hyperbolic groups are virtually compact special in the sense of Haglund and Wise \cite{HW}. In particular the group $\G$ constructed in Proposition \ref{rips1}
is residually finite (cf.~\cite{wiseRF} and \cite{HW}).
It also follows from Agol's theorem, via Proposition 3.6 of \cite{pavel},
that virtually compact special hyperbolic groups are good in the sense of
Serre \cite{serre}, meaning that for every finite $\Z G$-module $M$ and $p\ge 0$, the map
$H^p(\widehat{G},M)\to H^p(\widehat{G},M)$ induced by the inclusion 
of $G$ into its profinite completion $G\hookrightarrow \widehat{G}$, is an
isomorphim. We shall need this remark in our proof of Theorem \ref{t:thmC}.
\end{remark}

\subsection{Profinite completions and Grothendieck Pairs}\label{s:pf}
Throughout, $\widehat{G}$ denotes
the profinite completion of a group $G$. By definition, $\widehat{G}$
is the inverse limit of the directed system of
finite quotients of $G$. 
The natural
map $G\to \hat G$ is injective if and only if $G$ is residually finite.
A  Grothendieck pair is a monomorphism $u: P\hookrightarrow G$ 
of residually finite groups such
that $\hat u :\hat P\to\hat G$ is an isomorphism but $P$ is not isomorphic to 
$G$. The existence of non-trivial Grothendieck pairs of finitely
presented groups was established by Bridson and Grunewald in \cite{BG} following an earlier breakthrough by Platonov and Tavgen in the finitely generated case \cite{PT}.

The following criterion  plays a central role in
\cite{PT},  \cite{BL} and \cite{BG}.

\begin{proposition}\label{l:PT} Let $1\to N\to H\to Q\to 1$ be an exact sequence of groups with fibre product $P$. 
Suppose $H$ is finitely generated, $Q$ is finitely presented, and $H_2(Q,\Z)=0$. If
$Q$ has no proper subgroups of finite index, then the inclusion
$P\hookrightarrow H\times H$ induces an isomorphism of profinite completions.
\end{proposition}

It follows easily from the universal property of profinite completions
that if  $\widehat{G}\cong\widehat{H}$ then $G$ and $H$ have the same finite images. For finitely
generated groups, the converse is true  \cite[pp.~88--89]{RZ}.
Asking for $P\hookrightarrow G$ to be a Grothendieck pair is  more
demanding than asking simply that  there should be an abstract isomorphism
$\widehat{P}\cong\widehat{G}$. To see this we consider a pair of
groups constructed by Gilbert Baumslag \cite{gilb-exs}.

\begin{proposition} Let $G_1=(\Z/25)\rtimes_{\alpha}\Z$ and let $G_2=(\Z/25)\rtimes_{\alpha^2}\Z$, where $\alpha\in{\rm{Aut}}(\Z/25)$ is multiplication by $6$.
\begin{enumerate}
\item $G_1\not\cong G_2$.
\item $\widehat{G_2}\cong \widehat{G_1}$.
\item No homomorphism $G_1\to G_2$ or $G_2\to G_1$ induces an isomorphism
between $\widehat{G_1}$ and $\widehat{G_2}$. 
\end{enumerate}
\end{proposition}

\begin{proof} For $i=1,2$, let
$A_i$ be the unique $\Z/25 <G_i$. Each monomorphism
$\phi:G_1\to G_2$ restricts to an isomorphism
$\phi: A_1\to A_2$ and induces a monomorphism $G_1/A_1\to G_2/A_2$.
This last map cannot be an isomorphism: choosing a generator $t\in\Z<G_1$ so that $t^{-1}at=a^6$
for every $a\in A_1$  (writing the group operation in $A$ multiplicatively), we have $\phi(t)^{-1}\alpha \phi(t)=\alpha^6$
for all $\alpha\in A_2$, whereas $\tau^{-1}\alpha \tau=\alpha^{\pm 11}$
for  each $\tau\in G_2$ such that $\tau A_2$ generates $G_2/A_2$.  
This proves (1). 

With effort, one can prove that $G_1$ and $G_2$ have the same finite quotients
by direct argument after noting that any finite quotient $G_i\to Q$
that does not kill $A_i$ must factor through $G_i\to A_i\rtimes(\Z/5k)$ for some $k$.
Baumslag \cite{gilb-exs} gives a more elegant  and instructive proof of (2).

As $G_1$ and $G_2$ are residually finite, any map $\phi:G_1\to G_2$
that induces an isomorphism $\widehat{\phi}:\widehat{G}_1\to \widehat{G}_2$
must be a monomorphism.
The argument in the first paragraph shows in this case the image of
$\phi$ will be a proper subgroup of finite index  in $G_2$. If the index
is $d>1$, then the image of $\widehat{\phi}$ will have index $d$ in $ \widehat{G}_2$.
The same argument is valid with the roles of  $G_1$ and $G_2$ reversed, so (3)
is proved.
\end{proof}

\subsection{Biautomatic groups} 
The theory of automatic groups grew out of investigations into the algorithmic 
structure of Kleinian groups by Cannon and Thurston, and it was developed thoroughly
in the book by Epstein {\em{et al.}} \cite{epstein}; see also \cite{BGSS}. Let
$G$ be a group with finite generating set $A$ and let $A^*$ be the 
set of all finite words in the alphabet $A^{\pm 1}$. An
{\em{automatic structure}} for $G$ is determined by a normal
form $\mathcal{A}_G=\{\sigma_g \mid g\in G\}\subseteq A^*$ such that $\sigma_g=g$ in $G$. 
This normal form is required to satisfy two conditions: first, $\mathcal{A}_G\subset A^*$
must be a {\em{regular language}}, i.e. the accepted language of a finite state automaton;
and second, the edge-paths in the Cayley graph $\mathcal{C}(G,A)$ that begin at
$1\in G$ and are labelled by the words $\sigma_g$ must satisfy the following {\em{fellow-traveller condition}}: there is a constant $K\ge 0$
such that for all $g,h\in G$ and all integers $t\le \max\{|\sigma_g|, |\sigma_h|\}$,
$$
d_A(\sigma_g(t), \sigma_h(t))\le K \, d_A(g,h),
$$
where $d_A$ is the path metric on  $\mathcal{C}(G,A)$ in which each edge has length $1$,
and $\sigma_g(t)$ is the image in $G$ of the initial subword of length $t$ in $\sigma_g$.

A group is said to be {\em{automatic}} if it admits an automatic structure. If $G$
admits an automatic structure with the additional property that for  all integers $t\le \max\{|\sigma_g|, |\sigma_h|\}$,
$$
d_A(a.\sigma_g(t), \sigma_h(t))\le K \, d_A(ag,h),
$$
for all $g,h\in G$ and $a\in A$, then $G$ is said to be {\em{biautomatic}}. 
Biautomatic
groups were first studied by Gersten and Short \cite{GS}. Automatic and biautomatic groups form two of the most
important classes  studied in connection with notions of non-positive curvature in group theory; see \cite{mrb:camb} for a recent survey.

The established subgroup theory of biautomatic groups is considerably richer than that of
automatic groups. Biautomatic groups have a solvable 
conjugacy problem, whereas this is unknown for automatic groups. Groups in both
classes enjoy a rapid solution to the word problem, and have classifying 
spaces with finitely many cells in each dimension. The
isomorphism problem is open in both classes. No example has been found to distinguish between the two classes. 

\subsection{Some groups without finite quotients}\label{s:groups}
Graham Higman \cite{higman2} 
gave the first example of a finitely presented group that has no
non-trivial finite quotients. Many others have been discovered since, including 
the group %\begin{equation}\label{Bp}
$$ 
B_p = \< a, b, \alpha, \beta \mid ba^{-p}b^{-1}a^{p+1},\, \beta\alpha^{-p}\beta^{-1}\alpha^{p+1},\, 
[bab^{-1},a]\beta^{-1},\, [\beta\alpha\beta^{-1},\alpha]b^{-1}\>.
$$ %\end{equation}
This presentation is aspherical for $p\ge 2$; see \cite{BG}.
$B_2$ is a quotient of the 4-generator finitely presented group $H$ that Baumslag and Miller  concocted in \cite{gilb-chuck}. There is a surjection 
$H\to H\times H$, from which it follows that $H$ (and hence $B_2$)
cannot map onto a non-trivial finite group: for if $Q$ were
such a group, then the number
of distinct epimorphisms would satisfy
$|{\rm{Epi}}(H,Q)| < |{\rm{Epi}}(H\times H,Q)|$, which is nonsense if
$H$ maps onto $H\times H$.

\section{Proof of Theorem \ref{t:thm1} and Corollary \ref{c:cor}}\label{s:homology}

The proof of the following lemma is based on similar arguments in \cite{BDH} and \cite{BDM}.

\begin{lemma}\label{l:simult}
Let $\Pi$ be a property of groups that is inherited by direct limits
and suppose that every finitely presented group $G$ can be embedded in a 
finitely presented group $G_\Pi$ that has property $\Pi$.
Let $\Pi'$ be a second such property.   
Then there exists a  group $U^\dagger = K\rtimes\Z$  such that
\begin{enumerate}
\item $U^{\dagger}$ is finitely presented;
\item $U^{\dagger}$ contains an isomorphic copy of every finitely presented group;
\item $K$ has property $\Pi$ and property $\Pi'$.
\end{enumerate}
\end{lemma}

\begin{proof} Let $U_0$ be a finitely presented 
group that contains an isomorphic copy of every finitely presented group. The
existence of such groups was established by  Higman \cite{higman1}. 
By hypothesis, there is a finitely presented group $V$ that contains $U_0$ and has property $\Pi$, and there is a finitely presented group $W$ that contains $V$ and has property $\Pi'$. Consider the following chain of embeddings, where 
the existence of the embedding into $U_1 \cong U_0$ comes from the universal
property of $U_0$,
\begin{equation}\label{e1}
U_0 < V < W < U_1.
\end{equation}
We fix an isomorphism $\phi: U_1 \to U_0$ and define
$U^{\dagger}$ to be the ascending HNN extension  $(U_1, t \mid t^{-1}ut = \phi(u) \,
\forall u\in U_1)$. Let $K$ be the normal closure of $U_1$ in $U^{\dagger}$
and note that this is the kernel of the natural retraction $U^{\dagger}\to \<t\>$.
Note too that $t^{-i}U_1t^i < U_0$ for all positive integers $i$. It follows that
for each positive integer $d$, we can express $K$ as an
ascending union 
$$
K = \bigcup_{i\ge d} t^i U_1t^{-i}  = \bigcup_{i\ge d-1} t^i U_0 t^{-i}.
$$
From (\ref{e1}) we deduce that $K$ is the direct limit of each of the
ascending unions  $\bigcup_i t^iVt^{-i}$ and $\bigcup_i t^iWt^{-i}$.
The first union has property $\Pi$, while the
second has property $\Pi'$.
\end{proof}

\medskip

\subsection{Proof of Theorem \ref{t:thm1}} Every finitely presented group can be embedded
in a finitely presented group that has no finite quotients; see \cite{mb:embed}
for explicit constructions. And it is proved in the \cite{BDH} that every finitely
presented group can be embedded in a finitely presented acyclic group. It
is clear that having no non-trivial finite quotients is preserved under passage to direct
limits, and acyclicity is preserved because homology commutes with direct limits. 
Thus Lemma \ref{l:simult} provides us with a finitely presented group
$U^{\dagger} = K\rtimes \Z$ such that $K$ is acyclic and has no non-trivial
finite quotients. 

Let $B$ be a finitely presented acyclic group that has no non-trivial finite quotients and let
$\tau\in B$ be an element of infinite order -- we can take $B$ to be
$B_p$ from Section \ref{s:groups} for example.
Let $U = U^{\dagger}\ast_C B$ be the amalgamated free product in which $\<\tau\>$
is identified with   $C:=1\times \Z < U^{\dagger}$.

As $K$ is acyclic, by Lemma \ref{l:acylic-fibre}, $ U^\dagger\to C$
induces an isomorphism $H_*(U^\dagger, \Z) \cong H_*(C,\Z)$. In particular,
$H_n(U^\dagger, \Z) = 0$ for $n\ge 2$, and in the Mayer-Vietoris 
sequence for $U = U^{\dagger}\ast_C B$ the only potentially non-zero terms are
$$
0\to H_2(U,\Z) \to H_1(C,\Z) \to H_1(U^\dagger,\Z) \oplus H_1(B,\Z) 
\to H_1(U,\Z) \to 0.
$$
$H_1(B,\Z)=0$ and $H_1(C,\Z) \to H_1(U^\dagger,\Z)$ is an isomorphism, so we
deduce that $U$ is acyclic.

Each subgroup of finite index $S<U$ will intersect both $U^\dagger$ and $B$ in 
a subgroup of finite index. Since neither has any proper subgroups of finite index, 
$S$ must contain both $U^\dagger$ and $H$. Hence $S=U$.
\qed

\subsection{Proof of Corollary \ref{c:cor}} 
Theorem E of \cite{BDM} (see also \cite{cfm}) states that if 
$\mathcal{A}=(A_n)$ is as described in Corollary \ref{c:cor}
then there is a finitely generated, recursively presented
group $G_{\mathcal{A}}$ with $H_n(G_{\mathcal{A}},\Z)\cong A_n$ for all $n\ge 1$.

By the Higman Embedding Theorem \cite{higman1}, $G_{\mathcal{A}}$ can be embedded in 
the universal finitely presented group $U$ constructed in the preceding proof.
We form the amalgamated free product of two copies of $U$ along $G_{\mathcal{A}}$,
 $$
 Q_{\mathcal{A}}:=U\ast_{G_{\mathcal{A}}} U.
 $$
Note that because $G_{\mathcal{A}}$ is finitely generated, $Q_{\mathcal{A}}$
is finitely presented. As in the preceding proof,
since the  factors of the amalgam have no proper
subgroups of finite index, neither does $ Q_{\mathcal{A}}$.

The Mayer-Vietoris sequence for this amalgam yields, for all $n\ge 2$, an exact sequence
(where the $\Z$ coefficients have been suppressed):
$$
 H_n U\oplus H_n U\to H_n Q_{\mathcal{A}}  \to H_{n-1} G_{\mathcal{A}}\to H_{n-1}U \oplus H_{n-1}U .
$$
Thus, since $U$ is acyclic,  $H_nQ_{\mathcal{A}} \cong H_{n-1} G_{\mathcal{A}}
\cong A_{n-1}$ for all $n\ge 2$.
\qed 

\smallskip

Theorem E in \cite{BDM} is complemented by a number of ``untangling results" which 
avoid the untangled condition that appears in that theorem  and in our Corollary \ref{c:cor}.
The following is a special case of what is established in the proof of 
\cite[Theorem G]{BDM}.

\begin{proposition} 
For every recursively presented abelian group $A$, there exists a finitely generated, recursively presented 
group $G$ such that $H_1(G,\Z)=0$ and $H_2(G, \Z)\cong A$.
\end{proposition}

Exactly as in the proof of Corollary \ref{c:cor}, we deduce:

\begin{corollary}\label{c:H3}
For every recursively presented abelian group $A$, there exists a finitely presented 
group $Q_A$ with no proper subgroups of finite index
such that $H_1(Q_A,\Z)=H_2(Q_A,\Z)=0$ and $H_3(Q_A, \Z)\cong A$.
\end{corollary}

\section{Proof of Theorem \ref{t:perfect}}
The seed of undecidability that we need in Theorem \ref{t:perfect} comes from 
the following construction of Collins and Miller \cite{CM}.

\begin{theorem}\cite{CM}\label{t:cm} 
There is an integer $k$, a finite set $X$ and a recursive sequence $(R_n)$ of finite
sets of words in the letters $X^{\pm{1}}$ such that:
\begin{enumerate}
\item $|R_n| =k$ for all $n$, and $|X| < k$;
%\item each of the groups $Q_n =\<X \mid R_n\>$ is torsion-free;
\item all of the groups $Q_n \cong \<X \mid R_n\>$ are perfect;
\item there is no algorithm that can determine which of the
 groups $Q_n $ are trivial;
\item when $Q_n$ is non-trivial, the presentation $\mathcal{Q}_n\equiv 
\<X \mid R_n\>$ is aspherical.
\end{enumerate}
\end{theorem}

We apply our modified version of the Rips algorithm (Proposition \ref{rips1})
to the presentations $\mathcal{Q}_n$ from Theorem \ref{t:cm} to obtain a recursive sequence of finite presentations $(\mathcal{P}_n)$ of perfect groups $(\Gamma_n)$. By
applying the algorithm from Proposition \ref{p:tilde} to these presentations
we obtain a recursive sequence of finite presentations $(\widetilde{\mathcal{P}}_n)$ for the 
universal central extensions $(\widetilde{\G}_n)$.
By Proposition \ref{p:uce}(3), $\widetilde{\G}_n$ is perfect. We define $G_n=
\widetilde{\G}_n\times \widetilde{\G}_n$, with the obvious presentation $\mathcal{E}_n$
derived from $\mathcal{P}_n$.  

In more detail, with the notation established in Proposition \ref{rips1} and
Proposition \ref{p:tilde}, if $\mathcal{Q}_n \equiv \< X\mid R_n\>$
then $\mathcal{P}_n = \< X, a_1, a_2 \mid R_n\cup V_n\>$ and
$\widetilde{\mathcal{P}}_n = \< X, a_1, a_2 \mid \overline{R_n\cup V_n}\>$
while 
$$
\mathcal{E}_n \equiv \< X_1, X_2, a_{11}, a_{12},a_{21}, a_{22}
\mid
C,\, S_{1,n},\, S_{2,n}\>,
$$
where $X_1$ and $X_2$ are two copies of $X$ corresponding to the two factors
of $\widetilde{\G}_n\times \widetilde{\G}_n$ and $C$ is a list of commutators
forcing each $x_1\in X_1\cup \{a_{11}, a_{12}\}$ to 
commute with each $x_2\in X_2\cup\{a_{21}, a_{22}\}$, and $S_{i,n}\ (i=1,2)$
is the set of words obtained from $\overline{R_n\cup V_n}$ by replacing the 
ordered alphabet $ (X, a_1, a_2 ) $ with $ (X_i, a_{i1}, a_{i2} ) $.
Note that the generating set of $\mathcal{E}_n$ does not vary with $n$, and
nor does the cardinality of the set of relators.  
The map $X\cup\{a_1,a_2\}\to Q_n$ that kills $a_1$ and $a_2$ and is the identity on $X$
extends to give the composition of the 
universal central extension of $\G_n$ and the map $\G_n\to Q_n$ in the
Rips construction: 
\begin{equation}
\widetilde\G_n\to\G_n\to Q_n.
\end{equation}
By construction, the kernel of this map is the preimage $\widetilde{N}_n < \widetilde\G_n$
of $N_n=\<a_1,a_2\>< \G_n$. In particular, 
since the kernel of $\widetilde\G_n\to\G_n$ is finitely generated (isomorphic
to $H_2(\G_n,\Z)$), we see that $\widetilde{N}_n$ is finitely generated.
Thus for each $n$ we have a short exact sequence
\begin{equation}
1\to \widetilde{N}_n\to \widetilde{\G}_n\to Q_n\to 1
\end{equation}
with $\widetilde{N}_n$ finitely generated, $\widetilde{\G}_n$ finitely presented (indeed it
has a finite classifying space), and $Q_n$ as in Theorem \ref{t:cm}. 
In particular, since $Q_n$ is of type $F_3$, the 1-2-3 Theorem tells us that the
fibre product $P_n < \widetilde{\G}_n\times \widetilde{\G}_n=G_n$ associated to this
short exact sequence is finitely presentable. And Lemma \ref{l:genP} tells 
that $P_n$ is generated by 
$$\{(x,x),\, (a_1,1),\, (a_2,1),\, (r,1) \mid x\in X,\, r\in R_n\}.$$

At this stage, we have constructed the desired recursive sequence of pairs of groups
$(P_n\hookrightarrow G_n)_n$ with an explicit presentation for
the perfect group $G_n$ and
an explicit finite generating set for $P_n$. The inclusion $P_n\hookrightarrow
G_n$ is defined by  $(x,x)\mapsto x_1x_2,\, (a_i,1)\mapsto a_{1i}$ {\em etc.}
Our next task is to prove that there is no algorithm that can determine for which $n$
the group $P_n$ is perfect. 

\smallskip
\noindent{\underline{\em Claim:}} The 
recursively enumerable set $\{n\mid P_n\text{ is perfect }\}\subset\N$
is not recursive. 
\smallskip

The claim will follow if we can argue that $P_n$ is perfect if and only if 
$Q_n$ is the trivial group. If $Q_n=1$
then $P_n=G_n$, and we constructed $G_n$ to be perfect. If $Q_n\neq 1$,
then by Theorem \ref{t:cm}(4), the presentation $\mathcal{Q}_n$ is aspherical --
i.e.~the presentation 2-complex $K$ for $\mathcal{Q}_n$ is a classifying
space $K(Q_n,1)$. In this case, $H_2(Q_n,\Z)=H_2(K,\Z)$ is free abelian. 
As $H_1(Q_n,\Z)=H_1(K,\Z)=0$, the rank of  $H_2(Q_n,\Z)$ is $v_2-v_1$, where
$v_2$ is the number of generators on $\mathcal{Q}_n$
(1-cells in $K$) and $v_2$ is the number or relators (2-cells). 
Theorem \ref{t:cm}(1) tells us that $H_2(Q_n,\Z)\neq 0$, so 
we will be done if we can prove that $H_1(P_n,\Z)\cong H_2(Q_n,\Z)$.

From the 5-term exact sequence for $1\to \widetilde N_n \to \widetilde{\G}_n
\to Q_n\to 1$ we have
$$
H_2(\widetilde{\G}_n,\Z)\to H_2(Q_n,\Z)\to H_0(Q_n,H_1\widetilde N_n ) \to H_1(\widetilde{\G}_n,\Z).
$$
The first and last terms are zero, by Proposition \ref{p:uce}(2), 
so $H_2(Q_n,\Z)\cong H_0(Q_n,\widetilde N_n)$. 
On the other hand,   
from the 5-term exact sequence for $P_n = \widetilde N_n\rtimes \widetilde{\G}_n$ we have $H_0(\widetilde{\G}_n, H_1\widetilde N_n)
\cong H_1(P_n,\Z)$. As in Lemma \ref{l:sameH_0},
we observe that $H_0(\widetilde{\G}_n, H_1\widetilde N_n)= H_0(Q_n,  H_1\widetilde N_n)$, so
$H_1(P_n,\Z)\cong H_2(Q_n,\Z)$. This completes the proof of the Claim. 

\smallskip
In order to complete the proof of Theorem \ref{t:perfect}, we must explain why
$G_n$ is biautomatic and residually finite. 
First, Neumann and Reeves \cite{NR} proved that all finitely generated central
extensions of hyperbolic groups are biautomatic; $\G_n$ is hyperbolic and therefore
$\widetilde{\G}_n$ is biautomatic. And the direct product of two biautomatic groups is 
biautomatic, so $G_n$ is biautomatic. The residual finiteness of $\widetilde{\G}_n$
(and hence $G_n$) is a deeper fact, depending on the work of Wise and Agol:
we saw in Remark \ref{r:special} that ${\G}_n$ is residually finite
and good in the sense of Serre; if $A$ is a finitely generated abelian group
and $G$ is a finitely generated residually finite group that is good, then for
any central extension $1\to A\to E\to G\to 1$, the
group $E$ is residually finite -- see 
\cite[Section I.2.6]{serre} and \cite[Corollary 6.2]{pavel};
thus ${\G}_n$ is residually finite.
\qed

\section{Proof of Theorem \ref{t:thmC}}

We restate Theorem \ref{t:thmC}, for the convenience of the reader.

\begin{theorem}\label{t:thmC'} For every recursively presented abelian group $A$ there exists a Grothendieck pair
$P_A\hookrightarrow G_A$ where $G_A$  is a torsion-free biautomatic group that is residually finite, has a finite classifying space and is superperfect, while
$P_A$ is finitely generated with $H_2(P_A,\Z)\cong A$.
\end{theorem}

\begin{proof}
Corollary \ref{c:H3} provides us with a finitely presented group $Q$ that has
no finite quotients, with  $H_1(Q_A,\Z)=H_2(Q_A,\Z)=0$ and $H_3(Q_A,\Z)\cong A$.
As in the proof of Theorem \ref{t:perfect},
we apply Proposition \ref{rips1}  to obtain
a short exact sequence
$$
1\to N \to \G_A \overset{p}\to Q_A\to 1
$$
where $\G_A$ is a metric small cancellation group and $N$ is finitely generated.
The argument in the final 
two paragraphs of the proof of Theorem \ref{t:perfect} shows that
the universal central extension $\tilde \G_A$ is  biautomatic (by \cite{NR})
and that it is residually finite (by virtue of the connection between specialness
and goodness in the sense of Serre). The asphericity of the small 
cancellation presentation for $\G_A$ implies, in the light of Proposition \ref{p:uce},
that $\tilde\G_A$ has a finite classifying space $K(\tilde\G_A,1)$.
 
Let $\eta:\tilde\G_A\to Q_A$ be the composition of the central extension
$\tilde{\G}_A\to \G_A$ and $p:\G_A \to Q_A$ and
let $P_A< G_A:=\tilde \G_A\times \tilde \G_A$ be the fibre product associated to the
short exact sequence 
\begin{equation}\label{e2}
1\to \tilde N\to \tilde\G_A\overset{\eta}\to Q_A\to 1.
\end{equation}
Lemma \ref{l:genP} assures us that $P_A$ is finitely generated.  
Thus  we will be done if we can show that $P_A\hookrightarrow G_A$ 
induces an isomorphism of profinite completions and that $H_2(P_A,\Z)\cong A$.
The first of these assertions is a special case of Lemma \ref{l:PT}, since
$\wh{Q}_A=1$ and $H_2(Q_A,\Z)=0$. 
The second assertion relies on a comparison of the LHS spectral sequences associated to (\ref{e2}) and $1\to \tilde N \to P_A\to \tilde\G_A\to 1$. The key
points are isolated in the lemmas in Section \ref{s:LHS}.
Using these lemmas, we conclude our argument as follows.

From Lemma \ref{l:sameH_0} we have
\begin{equation}\label{ee1}
H_0(\tilde\G_A, H_2\tilde N) = H_0(Q_A, H_2\tilde N),
\end{equation}
where the first group of coinvariants is for the action induced by conjugation in
$P_A$  and the second is induced by conjugation in $\tilde\G_A$. % (from (\ref{e2}).

From Lemma \ref{l2} we have 
\begin{equation}\label{ee2} H_2(P_A,\Z)\cong H_0(\tilde\G_A, H_2N).
\end{equation}   
 Lemma \ref{l3} applies  to the short exact sequence (\ref{e2}), yielding  $H_1\tilde N=0$.
And we {\em claim} that  Lemma \ref{l1} also applies to  (\ref{e2}), yielding 
$H_0(Q_A, H_2\tilde N)\cong H_3(Q_A,\Z)$. By combining this isomorphism  with (\ref{ee1}) and (\ref{ee2}),
we have  $H_2(P_A,\Z)\cong H_3(Q_A,\Z)$, as desired.

It remains to justify the claim  that Lemma \ref{l1} applies  to  (\ref{e2}).
Specifically, we must argue that  $\eta:\widetilde{\G}_A\to Q_A$
 induces the zero map on $H_3(-,\Z)$.
By construction, $\eta$ factors through $\widetilde{\G}_A\to \G_A$.
 The homology of $\G_A$ can be calculated
from the standard 2-complex of its aspherical
 presentation, so $H_k(\G_A,\Z)=0$ for all $k>2$,
and hence the composition
$$H_3(\widetilde{\G}_A,\Z)\to H_3(\G_A,\Z)\to H_3(Q_A,\Z)$$
is the zero map.
\end{proof}

\end{document}